\documentclass[11pt]{amsart}
\usepackage{amsmath,amsfonts,amssymb,graphicx}
\usepackage{hyperref}
\newcommand{\R}{{\mathbb R}}

\newcommand{\be}[1]{\begin{equation}\label{#1}}
\newcommand{\ee}{\end{equation}}
\renewcommand{\(}{\left(}
\renewcommand{\)}{\right)}
\newcommand{\ird}[1]{\int_{\R^d}{#1}\;dx}
\newcommand{\nrm}[2]{\|#1\|_{L^{#2}(\R^d)}}
\def\cprime{$'$}

\newcommand{\finprf}{\unskip\null\hfill$\;\square$\vskip 0.3cm}

\newtheorem{theorem}{Theorem}
\newtheorem{lemma}[theorem]{Lemma}

\newtheorem{corollary}[theorem]{Corollary}
\newtheorem{proposition}[theorem]{Proposition}
\newtheorem{remark}{\sl Remark}

\begin{document}
\title[Asymptotics of fast diffusion equations]{Fast diffusion equations: matching large time asymptotics by relative entropy methods}
\author[J. Dolbeault and G. Toscani]{Jean Dolbeault and Giuseppe Toscani}
\address[J. Dolbeault]
{Ceremade (UMR CNRS no. 7534), Universit\'e Paris-Dauphine, Place de Lattre de Tassigny, F-75775 Paris C\'edex 16, France}
\email{dolbeaul@ceremade.dauphine.fr}
\address[G. Toscani]
{University of Pavia
Department of Mathematics, Via Ferrata~1, 27100 Pavia, Italy}
\email{giuseppe.toscani@unipv.it}
\date{\today}
\begin{abstract}
A non self-similar change of coordinates provides improved matching asymptotics of the solutions of the fast diffusion equation for large times,
compared to already known results, in the range for which Barenblatt solutions have a finite second moment. The method is based on relative
entropy estimates and a time-dependent change of variables which is determined by second moments, and not by the scaling corresponding to the
self-similar Barenblatt solutions, as it is usually done.
\end{abstract}
\keywords{Fast diffusion equation; porous media equation; Barenblatt solutions;
Hardy-Poincar\'e inequalities; large time behaviour; second moment; asymptotic expansion;
 intermediate asymptotics; sharp rates; optimal constants}
\def\subjclassname{\it\textup{2010} AMS subject classification}
\subjclass[2010]{35B40; 35K55; 39B62}
\maketitle
\thispagestyle{empty}

\section{Introduction and main results}\label{Sec:Intro}

Consider on $\R^d$ the fast diffusion equation \be{Eqn} \frac{\partial v}{\partial \tau}+\nabla\cdot\(v\,\nabla v^{m-1}\)=0 \ee for some
$m\in(m_c,1)$ with $m_c:=(d-2)/d$. Assume that the initial data is a given nonnegative function $v_0$ in $L^1(\R^d)$. It is well known (see for
instance~\cite{Ba52}) that the large time behavior of the solution is captured by the \emph{Barenblatt solutions} given for any $(\tau,y)\in\R^+\times\R^d$ by
\[
\mathsf B(\tau,y):=(1+\tau)^{-\frac 1{m-m_c}}\(D+\tfrac 1{2\,d\,(m-m_c)}(1+\tau)^{-\frac 2{d\,(m-m_c)}}\,|y|^2\)^{-\frac{1}{1-m}}_+
\]
where $D>0$ is determined by the condition $\nrm{v_0}1=\nrm{\mathsf B(\tau,\cdot)}1$. Using entropy methods, it has been established in
\cite{BBDGV} how a linearized problem involving the relative entropy and the relative Fisher information determines the best rate of convergence
towards the Barenblatt solution. The note \cite{BDGV} is devoted to a refinement of the estimates in which the dependence on $v_0$ is clarified
and the precise value of the best possible rate of convergence is computed in terms of a spectral gap of the linearized operator associated to the
relative entropy and the relative Fisher information for all values of $m<1$. By taking advantage of the translation invariance, it is moreover
possible to impose that the solution evolves in the orthogonal of the eigenspace associated to the first non-zero eigenvalue of the linearized
operator for $m\in(m_1,1)$ with $m_1:=(d-1)/d$, thus providing an improved rate of convergence. The corresponding conserved quantity is the center
of mass, while the generators of the eigenspace are the derivatives of the Barenblatt solution with respect to each of the coordinates. There is
no other conserved quantity known, so that further improvements cannot be achieved directly by this method.

One may however notice that the eigenspace corresponding to the second non-zero eigenvalue in the range $m\in(m_1,1)$ is generated by the
infinitesimal dilation of the Barenblatt solution. It is therefore natural to try to adjust the Barenblatt solution by a scaling. This can be done
by taking a time-dependent change of variables where the scale is determined by the solution itself and not anymore by its asymptotic,
self-similar behavior, thus providing improved convergence rates. Asymptotically, we will recover the self-similar profile, but with a better matching. There is a price to pay: the rescaled equation has a time-dependent coefficient, which converges to a constant. From the point of view of the entropy -- entropy production inequality, however, nothing is changed, which is the main observation of this paper.

Let $\widetilde m_1:=d/(d+2)$. In the range $m\in(\widetilde m_1,m_1)$, the infinitesimal dilation
of the Barenblatt solution generates the eigenspace corresponding to the first non-zero eigenvalue. Our time-dependent change of variables
therefore improves on the rate of convergence for any $m\in(\widetilde m_1,m_1)$, and also for any $m\in(m_1,1)$ if the center of mass is chosen at the origin.

The reader interested in understanding the heuristics of our approach is invited to go directly to Section~\ref{Sec:RelativeEntropy}. In the
remainder of this section, we will give a precise statement of our main result and some additional references.

\medskip Define the mass and the center of mass of $v_0$ respectively by
\[
M:=\int_{\R^d}v_0\;dy\quad\mbox{and}\quad x_0:=\frac 1M\int_{\R^d}y\,v_0\;dy\;.
\]
Consider the family of the Barenblatt profiles
\be{Eqn:Barenblatt}
B_\sigma(x):=\sigma^{-\frac d2}\(C_M+\tfrac 1\sigma\,|x|^2\)^\frac 1{m-1}\quad\forall\;x\in\R^d
\ee
where $\sigma$ is a positive real parameter and
\[
C_M:=\(\frac M{M_*}\)^{-\frac{2\,(1-m)}{d\,(m-m_c)}}\kern -4pt,\quad M_*:=
\ird{\(1+|x|^2\)^{\frac 1{m-1}}}= \pi^\frac d2\,\frac{\Gamma\big(\frac{d\,(m-m_c)}{2\,(1-m)}\big)}{\Gamma\big(\frac 1{1-m}\big)}\;.
\]
Notice that $\nrm{B_\sigma}1=M$ for any $\sigma>0$ and $B_\sigma$ is a solution of
\[
\nabla\cdot\left[B\(\sigma^{\frac d2(m_c-m)}\,\nabla B^{m-1}-2\,x\)\right]=0\;.
\]
Let us recall the definition of $m_c$, $m_1$, $\widetilde m_1$ and introduce the exponents $m_2$ and $\widetilde m_2$, for later use:
\begin{multline*}
m_c=\frac{d-2}d\;,\quad m_1=\frac{d-1}d\;,\quad\widetilde m_1=\frac d{d+2}\;,
\\m_2:=\frac{d+1}{d+2}\quad\mbox{and}\quad\widetilde m_2:=\frac{d+4}{d+6}\;,
\end{multline*}
which are such that $m_c<m_1<m_2<1$, $m_c<\widetilde m_1<\widetilde m_2$ and,
if $d\ge 2$, $\widetilde m_1\le m_1$ and $\widetilde m_2\le m_2$. For later purpose, it is also convenient to define
\be{Eqn:kappa}
K_M:=\ird{|x|^2\,B_1}=\frac{(1-m)\,\widetilde m_1}{m-\widetilde m_1}\,M\,C_M
\ee
for any $m\in(\widetilde m_1,1)$. Notice that $K_M$ is finite if $m>\widetilde m_1$.

If $v$ is a solution of \eqref{Eqn}, consider the time-dependent scale $R(\tau)$ defined by \be{Eqn:Scale} \frac{d\log
R}{d\tau}=2\(\frac{K_M}{\int_{\R^d}|y-x_0|^2\,v(\tau,y)\;dy}\)^{\frac d2(m-m_c)}\,,\quad R(0)=1\;. \ee
The justification of such a choice for
$\tau\mapsto R(\tau)$ will be made clear  in Section~\ref{Sec:RelativeEntropy}. Note that, in view of the fact that the quantity $\int_{\R^d}|y-x_0|^2\,v(\tau,y)\,dy$ is non-decreasing in time  along the solution of the fast diffusion equation~\eqref{Eqn} and because of its asymptotic behaviour, the time-dependent scale $R(\tau)$ is such that $\log R(\tau)$ is increasing from zero to infinity. We also define $\sigma$ as a function of $\tau$ by the condition
\[
\sigma(\tau):=\frac1{K_M\,R(\tau)^2}\int_{\R^d}|y-x_0|^2\,v(\tau,y)\;dy\;.
\]
As a consequence, the equation for $R$ can be rewritten as
\be{Eqn:sigma}
2\,\sigma^{-\frac d2(m-m_c)}=R^{\,1-d\,(1-m)}\,\frac{dR}{d\tau}
\ee
and we can define for any $x=y/R(t)\in\R^d$ the Barenblatt type solution $\mathcal B$ by
\[
\mathcal B(\tau,y):=\frac 1{R(\tau)^d}\,\mathfrak B\(\tau,\frac y{R(\tau)}\)\quad\mbox{where}\quad\mathfrak B(\tau,x):=B_{\sigma(\tau)}(x)\;.
\]
The difference of $\mathcal B$ with the Barenblatt solution $\mathsf B$ is that $\mathcal B$ depends on $\int_{\R^d}|y|^2\,v(\tau,y)\;dy$, so that
they are only asymptotically equivalent, as we shall see later. The point is that $\mathcal B$ provides a better asymptotic matching than $\mathsf B$. Our goal is
indeed to measure the rate of convergence of $v$ towards~$\mathcal B$. For this purpose, it is convenient to change variables and study the rate
of convergence of~$u$ defined by
\[
u(\tau,x)=R(\tau)^d\, v\(\tau,x_0+R(\tau)\,x\)
\]
towards $\mathfrak B$. Let us consider the \emph{relative entropy} of J. Ralston and W.I. Newmann defined in \cite{MR760591,MR760592} by
\[
\mathcal E(\tau):=\frac 1{m-1}\int_{\R^d}\left[u^m-\mathfrak B^m-m\,\mathfrak B^{m-1}\,(u-\mathfrak B)\right]\,dx\;.
\]
\begin{theorem}\label{Thm:Main} Assume that $m\in(\widetilde m_1,1)$, $d\ge 2$.
Let $v$ be a solution of~\eqref{Eqn} with initial datum $v_0\in L^1_+(\R^d)$ such
that $v_0^m$ and $y\mapsto|y|^2\,v_0(y)$ are integrable. With the above notations, we have
\[
\limsup_{\tau\to\infty}R(\tau)^{\gamma(m)}\,\mathcal E(\tau)<\infty\;,
\]
where $R(\tau)\sim\tau^{\frac 1{d\,(m-m_c)}}$ as $\tau\to\infty$ and
\[
\gamma(m)=\left\{\begin{array}{ll}
\frac{((d-2)\,m-(d-4))^2}{4\,(1-m)}&\mbox{if}\;m\in\(\widetilde m_1,\widetilde m_2\right]\;,\vspace{6pt}\cr
4\,(d+2)\,m-4\,d&\mbox{if}\;m\in\left[\widetilde m_2, m_2\right]\;,\vspace{6pt}\cr
4&\mbox{if}\;m\in[m_2,1)\,.\cr
\end{array}\right.
\]
\end{theorem}

Once a relative entropy estimate is known, it is possible to control the decay rate of $u-\mathfrak B$ in various norms, for instance in
$L^q(\R^d,\,dx)$ for
\[
q\ge \max\left\{1,\tfrac{2\,d\,(1-m)}{2\,(2-m)+d\,(1-m)}\right\}\;,
\]
or in $C^k$, by interpolation. Up to a change of variables, this also allows to prove decay rates of $v-\mathcal B$. See \cite{BBDGV} for more details.

Compared to the results of \cite{BDGV}, an improvement for any $m>\widetilde m_1$ has been obtained. The values obtained in \cite{BDGV} for
$\gamma(m)$ are indeed
\[
\gamma(m)=\left\{\begin{array}{ll}
2\,d\,m-2\,(d-2)&\mbox{if}\;m\in\(\widetilde m_1, m_1\right]\;,\vspace{6pt}\cr
2&\mbox{if}\;m\in[m_1,1)\;,\cr
\end{array}\right.
\]
except that in \cite{BDGV} the scale $R(\tau)$ is determined by the self-similar Barenblatt solutions (both scales are anyway equivalent as $\tau\to\infty$: see Lemma~\ref{Lem:Asymptotic}). Also see Figure~2 at the end of this paper for more details on $m\mapsto\gamma(m)$ in the setting of \cite{BDGV} compared to the results of Theorem~\ref{Thm:Main}.

Compared to other methods, it may look surprising that the scale $R(\tau)$ and, as a consequence, the coefficient $\sigma$ both depend on the solution $v$ of~\eqref{Eqn}. Asymptotically, as $\tau\to\infty$, $R(\tau)$ is equivalent to the scale given by the self-similar change of variables, but what has been gained is a better matching with the closest Barenblatt solution. The family of the Barenblatt solutions is globally invariant under scaling and, among all such solutions, there is one which is closer to our solution of the evolution equation: the one with the same second moment.

\medskip Convergence results of a suitably rescaled flow associated to \eqref{Eqn}
 towards an asymptotic profile has been established in \cite{MR586735} for $m>m_c$
 (also see for instance \cite{MR1977429}) and in \cite{BBDGV,Daskalopoulos-Sesum2006}
  for $m\le m_c$.
Getting rates of convergence beyond a simple interpolation between mass and uniform estimates has required the use of the relative entropies
introduced by J. Ralston and W.I. Newman in~\cite{MR760591,MR760592}. First results in this direction have been achieved in \cite{MR1777035} using
the entropy / entropy-production method of D. Bakry and M. Emery (also see \cite{MR1853037} for general diffusions and \cite{MR2065020} for an
overview) and in \cite{MR1940370} using sharp Gagliardo-Nirenberg interpolation inequalities. F.~Otto made the link with gradient flows with
respect to the Wasserstein distance in \cite{MR1842429}, and D. Cordero-Erausquin, B.~Nazaret and C.~Villani gave a proof of the corresponding
Gagliardo-Nirenberg inequalities using mass transportation techniques in \cite{MR2032031}. The condition $m\ge m_1$ was a strong limitation to
these first approaches. Gagliardo-Nirenberg inequalities indeed degenerate
into a critical Sobolev inequality for $m=m_1$, while the displacement convexity condition also requires $m\ge m_1$. For $m<m_1$, various
limitations appear. To work with Wasserstein's distance, it is crucial to have second moments bounded, which amounts to request $m>\widetilde m_1$
for the Barenblatt profiles; see for instance \cite{MR1982656,MR2126633}. Linearization of entropy estimates around the Barenblatt profiles has
been considered in \cite{MR1901093,MR1974458} for $m\in(m_c,1)$. In a certain sense, this is also the strategy in \cite{MR1982656,MR2126633}.
Integrability of the Barenblatt profiles means $m>m_c$. This condition has been removed in a series of recent papers (see
\cite{BBDGV-CRAS,BBDGV,BDGV,BGV}) together with a clarification of the strategy of linearization of the relative entropies, at least from the
point of view of functional inequalities. In this paper, we shall however restrict $m$ to the interval $(\widetilde m_1,1)$, for spectral
reasons that are explained in Section~\ref{Sec:HardyPoincare} and for the second moment to be well defined. For $m\le\widetilde m_1$, even with an appropriate definition a relative second moment, our method gives no improvement on the convergence rates because of the presence of the continuous spectrum.

Rescalings and convergence towards Barenblatt solutions, or \emph{intermediate asymptotics}, has not been the only issue of large time
asymptotics. We can for instance quote \cite{MR2133441} for a study (in the porous media case) of the time evolution of the second moment, and
\cite{MR2255281,MR2126633,MR2211152,MR2344717} for the search of improved convergence rates when moment conditions are imposed in the framework of
Wasserstein's or other Fourier based distances. The question of improved rates has been precisely formulated in \cite{MR2126633}, and solved in
\cite{BDGV} in the weighted $L^2$ framework that we shall use in this paper, as far as the first moment (position of the center of mass) is
concerned. The main contribution of this work is to explain how improvements based on the second moment can also be achieved.

This paper is organized as follows. In Section~\ref{Sec:RelativeEntropy}, we explain how faster convergence results can be achieved by introducing
an appropriate time-dependent rescaling, which is given by \eqref{Eqn:Scale} and not by the explicit dependence of the Barenblatt solutions.
Improved Hardy-Poincar\'e inequalities are established in Section~\ref{Sec:HardyPoincare}, using the spectral results of
\cite{MR1982656,MR2126633} and the spectral equivalence found in \cite{BDGV}. The large time behaviour of the solution is studied in
Section~\ref{Sec:SecondMoment}. The proof of Theorem~\ref{Thm:Main} is then completed in Section~\ref{Sec:Proof}. Further considerations on the
case $d=1$ and the limiting regime as $m\to 1_-$ are presented in the last section.
\newpage

\section{The relative entropy approach}\label{Sec:RelativeEntropy}

The result of Theorem~\ref{Thm:Main} is easy to understand using a time-dependent rescaling and the relative entropy formalism. Define the
function $u$ such that \be{Eqn:Time-dependent} v(\tau,y+x_0)=R^{-d}\,u(t,x)\;,\quad R=R(\tau)\;,\quad t=\tfrac 12\,\log R\;,\quad x=\frac yR \ee
where $v$ is a solution of~\eqref{Eqn} with initial datum $v_0\in L^1_+(\R^d)$. A simple computation shows that $u$ has to be a solution of
\be{Eqn2} \frac{\partial u}{\partial t}+\nabla\cdot\left[u\,\(\sigma^{\frac d2(m-m_c)}\,\nabla u^{m-1}-2\,x\)\right]=0\quad t>0\;,\quad
x\in\R^d\,, \ee with initial datum $u_0=v_0$ (we assume that $R$ is chosen such that $R(0)=1$) and $\sigma$ given by
\[
2\,\sigma^{-\frac d2(m-m_c)}=R^{\,1-d\,(1-m)}\,\frac{dR}{d\tau}\;,
\]
which is nothing else than \eqref{Eqn:sigma}. By virtue of the definition of $R(\tau)$, the new time $t = \frac 12\,\log R(\tau)$ increases monotonically from $0$ to $+\infty$. Consequently, the old and new times can be uniquely related, and $\tau$ can be expressed in terms of $t$ through the inverse function of $R$, so that $\tau = R^{-1}(e^{2t})$. Using this transformation and with
a slight abuse of notations, we shall consider from now on $\sigma$ as a function of $t$. It is
important to notice that, as long as $\frac{d\sigma}{dt}\neq 0$, the Barenblatt profile $B_\sigma$ \emph{is not} a solution of \eqref{Eqn2}, but
we may still consider the relative entropy
\[
\mathcal F_\sigma[u]:=\frac 1{m-1}\int_{\R^d}\left[u^m-B_\sigma^m-m\,B_\sigma^{m-1}\,(u-B_\sigma)\right]\,dx\;.
\]
Let us briefly sketch the strategy of our method before giving all details.

\medskip If we consider a solution of~\eqref{Eqn2} and compute the time derivative of the relative entropy, we find that
\be{Eqn:TwoTerms} \frac d{dt}\mathcal F_{\sigma(t)}[u(t,\cdot)]=\frac{d\sigma}{dt}\(\frac d{d\sigma}\,\mathcal
F_\sigma[u]\)_{|\sigma=\sigma(t)}\kern -5pt+\frac m{m-1}\ird{\(u^{m-1}-B_{\sigma(t)}^{m-1}\)\frac{\partial u}{\partial t}}\;. \ee Here comes the
main difference with previous works. As we shall see below in the proof of Lemma~\ref{Lem:ChoiceOfSigma} (also see Remark~\ref{Rem1}), the first term of the right hand side in
\eqref{Eqn:TwoTerms} involves
\[
\(\frac d{d\sigma}\,\mathcal F_\sigma[u]\)_{|\sigma=\sigma(t)}=\frac{m\,D(\sigma)}{1-m}\(\ird{|x|^2\,B_\sigma}-\ird{|x|^2\,u}\)_{|\sigma=\sigma(t)}
\]
where $D(\sigma):=\frac d{d\sigma}\big(\sigma^{d\,(m_c-m)/2}\big)$. When taking a time-dependent rescaling based on the self-similar variables, one
finds that $\sigma$ is constant in $t$, so that $\frac{d\sigma}{dt}=0$ and the term $\frac{d\sigma}{dt}\,\frac d{d\sigma}\,\big(\mathcal F_\sigma[u]\big)$ does not contribute. In our approach,
$\sigma$ depends on $t$ but can still be chosen so that this term does not show up either. It is indeed enough to require that 
\be{equal-moment}
\ird{|x|^2\,B_\sigma}=\ird{|x|^2\,u}\;, \ee
which amounts to ask that $R$ solves the ordinary differential equation~\eqref{Eqn:Scale}, to obtain that $\frac d{d\sigma}\,\big(\mathcal F_\sigma[u]\big)=0$. This
will be justified in the first step of our method, below (see Lemma~\ref{Lem:ChoiceOfSigma}).

In a second step, we shall use the fact that \be{Eqn:Standard} \frac d{dt}\mathcal
F_{\sigma(t)}[u(t,\cdot)]=-\frac{m\,\sigma(t)^{\frac d2(m-m_c)}}{1-m}\ird{u\left|\nabla\left[u^{m-1}-B_{\sigma(t)}^{m-1}\right]\right|^2}\;. \ee
From there on, the computation goes essentially as in \cite{BBDGV,BDGV}. For completeness, we will briefly reproduce it. However, with our choice
of $\sigma$, we gain an additional orthogonality condition which will be explicitly stated in the third step of the method: see
Lemma~\ref{Lem:OrthogonalityProperties}. This orthogonality condition is the crucial point (see Corollary~\ref{Cor:LinearizedProblem}) for
improving the rates in Theorem~\ref{Thm:Main}, compared to the results of \cite{BDGV}. Now let us give further details.

\subsection*{First step: choice of the scaling parameter}

For a while, we do not need to take into account the dependence of $\sigma$ in~$t$. The main idea of this paper is indeed to choose $\sigma$ in
terms of $u$ by minimizing $\sigma\mapsto\mathcal F_\sigma[u]$, so that
\[
\frac d{d\sigma}\,\mathcal F_\sigma[u]=0\;.
\]
\begin{lemma}\label{Lem:ChoiceOfSigma}
For any given $u\in L^1_+(\R^d)$ such that $u^m$ and $|x|^2\,u$ are both integrable,
if $m\in(\widetilde m_1,1)$, there is a unique $\sigma= {\sigma^*} >0$ which
minimizes $\sigma\mapsto\mathcal F_\sigma[u]$, and it is explicitly given by
\be{min}
{\sigma^*}=\frac 1{K_M}\,\ird{|x|^2\,u}\;.
\ee
For $\sigma=\sigma^*$, the Barenblatt profile $B_{{\sigma^*}}$ satisfies \eqref{equal-moment}.
\end{lemma}

In \eqref{min}, $K_M$ is the constant which has been defined in \eqref{Eqn:kappa}. The condition $m>\widetilde m_1$ guarantees that $B_\sigma^m$
is integrable and $K_M=\ird{|x|^2\,B_1}$ is finite.
\begin{proof} We have to minimize
\[
e(\sigma):=(1-m)\ird{B_\sigma^m}+m\ird{B_\sigma^{m-1}\,u}\;,
\]
which corresponds to the two $\sigma$-dependent terms in the expression of $\mathcal F_\sigma[u]$. Using the fact that
$B_\sigma^{m-1}(x)=\sigma^{d\,(1-m)/2}(C_M+|x|^2/\sigma)$ and $\ird u=M$, it is easy to see that
\begin{multline*}
e(\sigma)=(1-m)\,\sigma^{\frac d2\,(1-m)}\,\big[M\,C_M+K_M\big]\\+m\,\sigma^{\frac d2\,(1-m)}\,\(M\,C_M+\ird{\frac{|x|^2}\sigma\,u}\)
\end{multline*}
Collecting terms, we get
\[
e(\sigma)=\sigma^{\frac d2\,(1-m)}\,\big[M\,C_M+(1-m)\,K_M\big]\\+m\,\sigma^{\frac d2\,(1-m)-1}\ird {|x|^2\,u}\;.
\]
Hence, by optimizing on $\sigma>0$, we find that $\sigma=\sigma^*$ is given by
\[
\sigma^*=\frac{m\,(m-m_c)}{(1-m)\,\big[M\,C_M+(1-m)\,K_M\big]}\ird {|x|^2\,u}\;.
\]
Using \eqref{Eqn:kappa}, we observe that
\[
\frac{m\,(m-m_c)}{(1-m)\,\big[M\,C_M+(1-m)\,K_M\big]}=\frac 1{K_M}\;,
\]
which proves \eqref{equal-moment} and therefore \eqref{min}.\end{proof}

\begin{remark}\label{Rem1} To prove \eqref{equal-moment} directly, we may notice that ${\sigma^*}$ is determined by the condition
\[
0=e'(\sigma)=m\,(1-m)\ird{B_\sigma^{m-2}\,\frac{dB_\sigma}{d\sigma}\({B_\sigma}- u\)}\;.
\]
Hence we obtain
\[
0 = m\!\ird{\kern -3pt\({B_\sigma}- u\)\,\frac{dB_\sigma^{m-1}}{d\sigma}} =
 m\!\ird{\kern -3pt\({B_\sigma}- u\)\Big(C(\sigma)\,C_M + D(\sigma)\,|x|^2\Big)}\;,
\]
where $C(\sigma) = \frac d{d\sigma}\big(\sigma^{d\,(1-m)/2}\big)$, and $D(\sigma) = \frac d{d\sigma}\big(\sigma^{d\,(1-m)/2-1}\big)$. Taking into
account that both $B_\sigma$ and $u$ have the same mass $M$, we get \eqref{equal-moment}. \end{remark}

The dependence of $\sigma$ in $t$ when $u$ is a solution of~\eqref{Eqn2} has not been taken into account yet. The choice of
Lemma~\ref{Lem:ChoiceOfSigma} determines an ordinary differential equation for $R$ in terms of $\ird{|x|^2\,u(t,x)}$. Undoing the time-dependent
rescaling \eqref{Eqn:Time-dependent}, this equation is exactly \eqref{Eqn:Scale}. With the choice $R(0)=1$, we recall that $u_0=v_0$.

\medskip As already mentioned, the choice of $\sigma$ in Lemma~\ref{Lem:ChoiceOfSigma}
has a major interest. If we consider a solution of~\eqref{Eqn2} and compute the time derivative
of the relative entropy, we find that the first term of the right hand side in \eqref{Eqn:TwoTerms} drops so that
\[
\frac d{dt}\mathcal F_{\sigma(t)}[u(t,\cdot)]=\frac m{m-1}\ird{\(u^{m-1}-B_{\sigma(t)}^{m-1}\)\frac{\partial u}{\partial t}}
\]
and we are back to the usual computations in self-similar variables.

\subsection*{ Second step: the entropy / entropy production estimate}

According to the definition of $B_\sigma$, we know that
\[
2\,x=\sigma^{\frac d2(m-m_c)}\,\nabla B_\sigma^{m-1}\,.
\]
Using \eqref{Eqn2}, we obtain
\begin{multline*}
\frac m{m-1}\ird{\(u^{m-1}-B_{\sigma(t)}^{m-1}\)\frac{\partial u}{\partial t}}\\
=-\frac{m\,\sigma(t)^{\frac d2(m-m_c)}}{1-m}\ird{u\left|\nabla\left[u^{m-1}-B_{\sigma(t)}^{m-1}\right]\right|^2}\;,
\end{multline*}
thus proving \eqref{Eqn:Standard}. Let $w:=u/B_\sigma$ and observe that the relative entropy can be written as
\[
\mathcal F_\sigma[u]=\frac{m}{1-m}\int_{\R^d}\Big[w-1-\frac{1}{m}\big(w^m-1\big)\Big]\,B_\sigma^m\,dx\;.
\]
Define the \emph{relative Fisher information} by
\[
\mathcal I_\sigma[u]:=\int_{\R^d}\Big|\,\frac 1{m-1}\,\nabla\left[(w^{m-1}-1)\,B_\sigma^{m-1}\right]\Big|^2\,B_\sigma\,w\;dx\;.
\]
For a solution $u$ of \eqref{Eqn2}, we find that
\[
\frac d{dt}\mathcal F_{\sigma(t)}[u(t,\cdot)]=-\,m\,(1-m)\,\sigma(t)^{\frac d2(m-m_c)}\,\mathcal I_{\sigma(t)} [u(t,\cdot)]\quad\forall\;t>0\;.
\]
As in \cite[Lemma 3]{BBDGV} (also see \cite{BDGV}) we can estimate from below and above the entropy $\mathcal F_\sigma[u]$~by \be{Eqn:Entropy}
h^{m-2}\int_{\R^d}|f|^2\,B_\sigma^{2-m}\;dx\le \frac 2m\,\mathcal F_\sigma[u]\le h^{2-m}\int_{\R^d}|f|^2\,B_\sigma^{2-m}\;dx \ee where
$f:=(w-1)\,B_\sigma^{m-1}$, $h_1(t):=\mathrm{inf}_{\R^d}w(t,\cdot)$, $h_2(t):=\mathrm{sup}_{\R^d}w(t,\cdot)$ and $h:=\max\{h_2,1/h_1\}$. The fact
that $h$ is bounded for any $t>0$ is easy to prove by the Maximum Principle if $h(0)$ is finite. See for instance \cite{BBDGV} for more details.
Even if $h(0)$ is infinite, $h$ is anyway bounded for any $t>0$, large enough, when $m\in(m_c,1)$; see \cite[Theorem 1.2]{BV}. By \cite[Corollary 1]{BBDGV}, we
also know that \hbox{$\lim_{t\to\infty}h(t)=1$}.

According to \cite[Lemma 7]{BBDGV} (also see \cite{BDGV}) , the generalized Fisher information satisfies the bounds
\be{Eqn:Fisher}
\int_{\R^d}|\nabla f|^2\,B_\sigma\;dx\le [1+X(h)]\,\mathcal I_\sigma[u]+Y(h)\int_{\R^d}|f|^2\,B_\sigma^{2-m}\;dx
\ee
where\begin{eqnarray*}
&&h_1^{-1}\,h_2^{2\,(2-m)}\le h^{5-2m}=:1+X(h)\;,\\
&&d\,(1-m)\left[(\tfrac{h_2}{h_1})^{2\,(2-m)}-1\right]\le d\,(1-m)\,\big[h^{4\,(2-m)}-1\big]=:Y(h)\;.
\end{eqnarray*}Notice that $X(1)=Y(1)=0$.

\subsection*{Third step: orthogonality conditions}
To obtain decay rates of $t\mapsto\mathcal F_{\sigma(t)}[u(t,\cdot)]$ is now reduced to establish a relation between
$\int_{\R^d}|f|^2\,B_\sigma^{2-m}\;dx$ and $\int_{\R^d}|\nabla f|^2\,B_\sigma\;dx$. This is the purpose of the next section, but before let us
make a few additional observations on the properties of $f=B_\sigma^{m-2}\,u-B_\sigma^{m-1}$.
\begin{lemma}\label{Lem:OrthogonalityProperties}  Let $u$ be a solution of \eqref{Eqn2} and $f=B_\sigma^{m-1}\,(u/B_\sigma-1)$
where $\sigma=\sigma(t)$ is defined by \eqref{Eqn:sigma}. With these notations, the function $f$ has the following properties, for any $t>0$ :
\begin{enumerate}
\item[(i)] \emph{Mass conservation:} $\ird{f(t,x)\,B_\sigma^{2-m}}=0$ if $m>m_c$.
\item[(ii)] Position of the \emph{center of mass:} $\ird{x\,f(t,x)\,B_\sigma^{2-m}}=0$ if $m>(d-1)/(d+1)$.
\item[(iii)] Conservation of the \emph{second moment:} $\ird{|x|^2\,f(t,x)\,B_\sigma^{2-m}}=0$ if $m>\widetilde m_1$.
\end{enumerate}
\end{lemma}

Notice that Property (ii) has already been used in \cite[Theorem 7]{BDGV} to obtain improved rates of convergence for $m\in(m_1,1)$. Property
(iii) is new and arises from the fact that the change of variables \eqref{Eqn:Time-dependent} is chosen in Lemma~\ref{Lem:ChoiceOfSigma} by
imposing a condition on the moment and not according to the self-similar variables corresponding to the Barenblatt solutions. Restrictions on $m$
are such that mass, first and second moments are well defined for $B_\sigma$. As in~\cite{BDGV}, such conditions could be relaxed by considering
moment conditions on~$f$ only, or relative moment conditions; we will however not pursue in this direction, as such an approach does not improve
on the rates of convergence.

\begin{proof} It is straightforward to rewrite the conservation of mass along evolution into Property (i).

Equation \eqref{Eqn} being independent of $y$, it is clear that $\int_{\R^d}y\,v(\tau,y)\;dy=\int_{\R^d}y\,v_0(y)\;dy$ for any $\tau>0$. As a
consequence, the \emph{center of mass} of $u$ is located at $x=0$: $\ird{x\,u(t,x)}=0$, and so we get Property (ii).

Finally, Property (iii) is a direct consequence of \eqref{equal-moment}. \end{proof}

\section{Improved Hardy-Poincar\'e inequalities}\label{Sec:HardyPoincare}

When $M=M^*$ and $\sigma=1$, with the notations of Section~\ref{Sec:Intro}, the quantities $\int_{\R^d}|f|^2\,B_\sigma^{2-m}\;dx$ and
$\int_{\R^d}|\nabla f|^2\,B_\sigma\;dx$ involve various powers of $(1+|x|^2)$. On~$\R^d$, we shall therefore consider the measure
$d\mu_\alpha:={\mu}_\alpha\,dx$, where the weight ${\mu}_\alpha$ is defined by ${\mu}_\alpha(x):= (1+|x|^2)^{\alpha}$, with $\alpha=1/(m-1)<0$, and study the
operator
\[
\mathcal L_{\alpha,d}:=-{\mu}_{1-\alpha}\,\mathrm{div}\left[\,{\mu}_\alpha\,\nabla\cdot\,\right]
\]
on the weighted space $L^2\,(d\mu_{\alpha-1})$. The operator $\mathcal L_{\alpha,d}$ is such that
\[
\int_{\R^d}f\,(\mathcal L_{\alpha,d}\,f)\,d\mu_{\alpha-1}=\int_{\R^d}|\nabla f|^2\,d\mu_\alpha\;.
\]

Notice that in the range $m\in(m_c,1)$, that is for $\alpha\in(-\infty,-d/2)$, $d\mu_\alpha$ is a bounded positive measure. If additionally
$m>\widetilde m_1$, that is for $\alpha<-(d+2)/2$, then $d\mu_{\alpha-1}$ is a bounded measure and $\int_{\R^d}|x|^2\,d\mu_\alpha$ is also finite.

Let $\alpha_*:=-(d-2)/2$. Based on \cite{MR1982656,MR2126633,BBDGV-CRAS,BDGV}, we have the following result.
\begin{proposition}\label{Prop:Spectrum} The bottom of the continuous spectrum of
the operator $\mathcal L_{\alpha,d}$ on $L^2\,(d\mu_{\alpha-1})$ is $\lambda_{\alpha,d}^{\rm cont}:=(\alpha-\alpha_*)^2$.
Moreover, $\mathcal L_{\alpha,d}$ has some discrete spectrum only for $\alpha<\alpha_*$.
For $d\ge 2$, the discrete spectrum is made of the eigenvalues
\[
\lambda_{\ell k}=-2\,\alpha\,\(\ell+2\,k\)-4\,k\(k+\ell+\tfrac d2-1\)
\]
with $\ell$, $k=0$, $1$, \ldots provided $(\ell,k)\neq(0,0)$ and $\ell+2k-1<-(d+2\,\alpha)/2$. If $d=1$,
the discrete spectrum is made of the eigenvalues $\lambda_k=k\,(1-2\,\alpha-k)$ with $k\in{\mathbb N}\cap[1,1/2-\alpha]$.\end{proposition}
Let $\alpha_*:=-(d-2)/2$. The following result has been established in \cite{BBDGV}.
\begin{corollary}[Sharp Hardy-Poincar\'e inequalities]\label{Cor:SharpHardyPoincare}
Let $d\ge 2$. For any $\alpha\in(-\infty,\alpha_*)\cup(\alpha_*,0)$, there is a positive constant $\Lambda_{\alpha,d}$ such that
\be{gap}
\Lambda_{\alpha,d}\int_{\R^d}|f|^2\,d\mu_{\alpha-1}\leq \int_{\R^d}|\nabla f|^2\,d\mu_\alpha\quad\forall\;f\in L^2\,(d\mu_{\alpha-1})
\ee
under the additional condition $\int_{\R^d}f\,d\mu_{\alpha-1}=0$ if $\alpha<\alpha_*$. Moreover, for $d\ge 3$, the sharp constant $\Lambda_{\alpha,d}$ is given by
\[
\Lambda_{\alpha,d}=\left\{\begin{array}{ll}
\frac 14\,(d-2+2\,\alpha)^2&\mbox{if}\;\alpha\in\left[-\frac{d+2}{2},\alpha_*\right)\cup(\alpha_*,0)\;,\vspace{6pt}\cr
-\,4\,\alpha-2\,d&\mbox{if}\;\alpha\in\left[-d,-\frac{d+2}{2}\right)\;,\vspace{6pt}\cr
-\,2\,\alpha&\mbox{if}\;\alpha\in(-\infty,-d)\,.\cr
\end{array}\right.
\]
For $d=2$, Inequality \eqref{gap} holds for all $\alpha<0$, with the corresponding values of the best constant $\Lambda_{\alpha,2}=\alpha^2$ for
$\alpha\in [-2,0)$ and $\Lambda_{\alpha,2}=-2\,\alpha$ for $\alpha\in (-\infty,-2)$. For $d=1$, \eqref{gap} holds, but the values of
$\Lambda_{\alpha,1}$ are given by $\Lambda_{\alpha,1}=-2\,\alpha$ if $\alpha<-1/2$ and $\Lambda_{\alpha,1}=(\alpha-1/2)^2$ if $\alpha\in[-1/2,0)$.
\end{corollary}
The constant $\Lambda_{\alpha,d}$ is determined by the spectral gap and corresponds either to the lowest positive eigenvalue, $\lambda_{1,0}$ or $\lambda_{0,1}$, or to the bottom of the continuous spectrum, $\lambda_{\alpha,d}^{\rm cont}:=\frac 14(d+2\,\alpha-2)^2$ (see Fig.~1).\medskip

With additional orthogonality conditions, one improves on the spectral gap in the range for which discrete spectrum exists. A first result in this
direction has been achieved in \cite{BDGV} for solutions with center of mass at the origin. Here we give a refined version of it, by going to the
next order, that is, by considering functions with zero moments up to order two.
\begin{corollary}[Improved Hardy-Poincar\'e inequalities]\label{Cor:ImprovedHardyPoincare}
Under the assumptions of Corollary~\ref{Cor:SharpHardyPoincare}, if $\alpha<-(d+2)/2$ and $f\in L^2\,(d\mu_{\alpha-1})$ is such that
\[
\int_{\R^d}f\,d\mu_{\alpha-1}=0\;,\quad\int_{\R^d}x\,f\,d\mu_{\alpha-1}=0\quad\mbox{and}\quad\int_{\R^d}|x|^2\,f\,d\mu_{\alpha-1}=0\;,
\]
then \eqref{gap} holds for any $d\ge 3$ with
\[
\Lambda_{\alpha,d}=\left\{\begin{array}{ll}
\frac 14\,(d-2+2\,\alpha)^2&\mbox{if}\;\alpha\in\left[-\frac{d+6}2,-\frac{d+2}2\right)\;,\vspace{6pt}\cr
-\,8\,\alpha-4\,(d+2)&\mbox{if}\;\alpha\in\left[-(d+2)\,,-\frac{d+6}2\right]\;,\vspace{6pt}\cr
-\,4\,\alpha&\mbox{if}\;\alpha\in(-\infty,-(d+2)]\,.\cr
\end{array}\right.
\]
For $d=2$, $\Lambda_{\alpha,2}=\alpha^2$ for $\alpha\in [-4,-2)$ and $\Lambda_{\alpha,2}=-4\,\alpha$ for $\alpha\in (-\infty,-4)$.
\end{corollary}
The constant $\Lambda_{\alpha,d}$ is now determined either by the lowest of the two eigenvalues, $\lambda_{1,1}$ or $\lambda_{2,0}$, or by the bottom of the continuous spectrum, $\lambda_{\alpha,d}^{\rm cont}$ (see Fig.~1), since the components corresponding to the eigenspaces associated to $\lambda_{1,0}$ and $\lambda_{0,1}$ have been removed.\medskip

\begin{figure}[ht]\begin{center}\includegraphics[height=10cm]{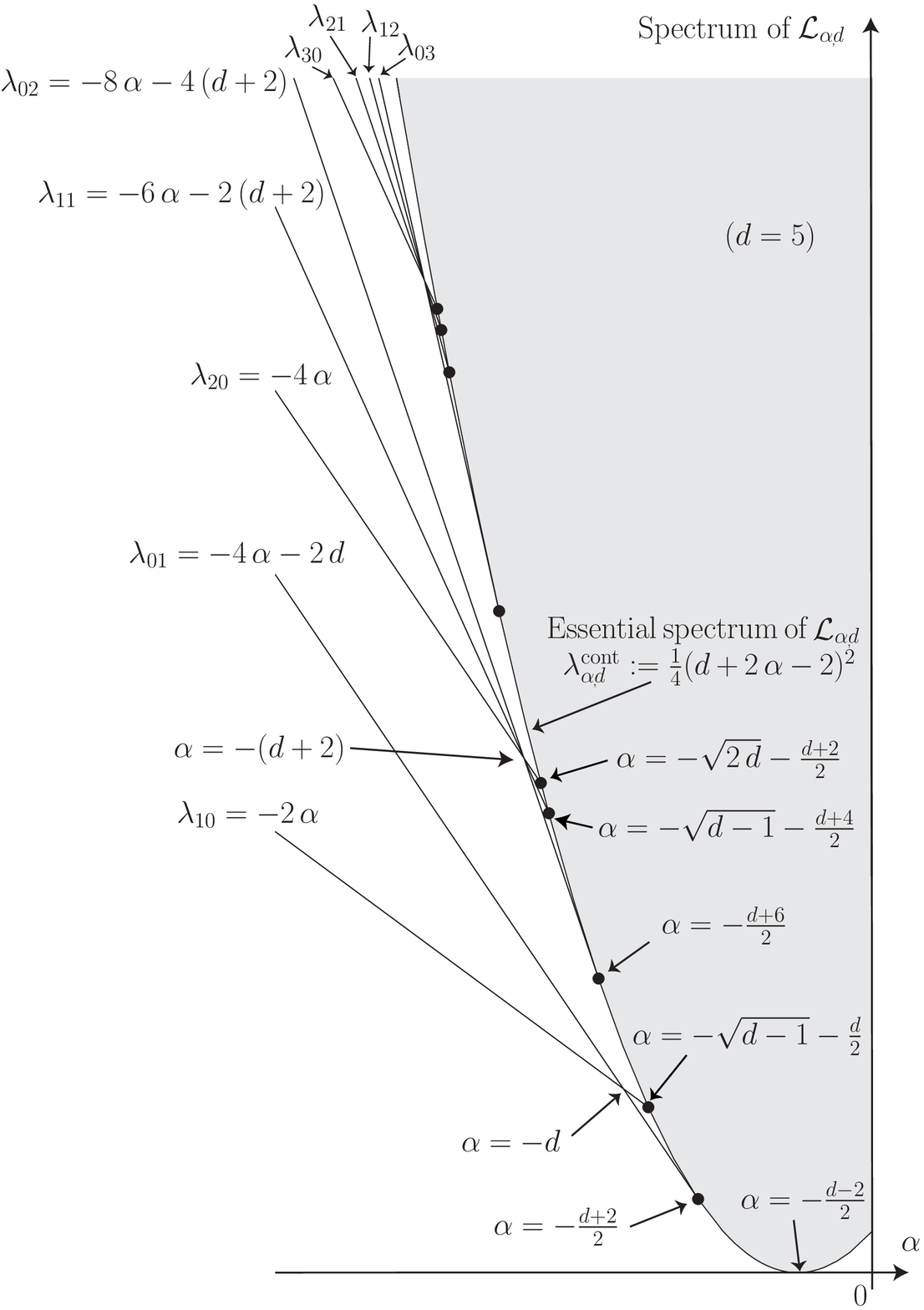}\hspace*{-6.1cm}\includegraphics[height=10.25cm]
{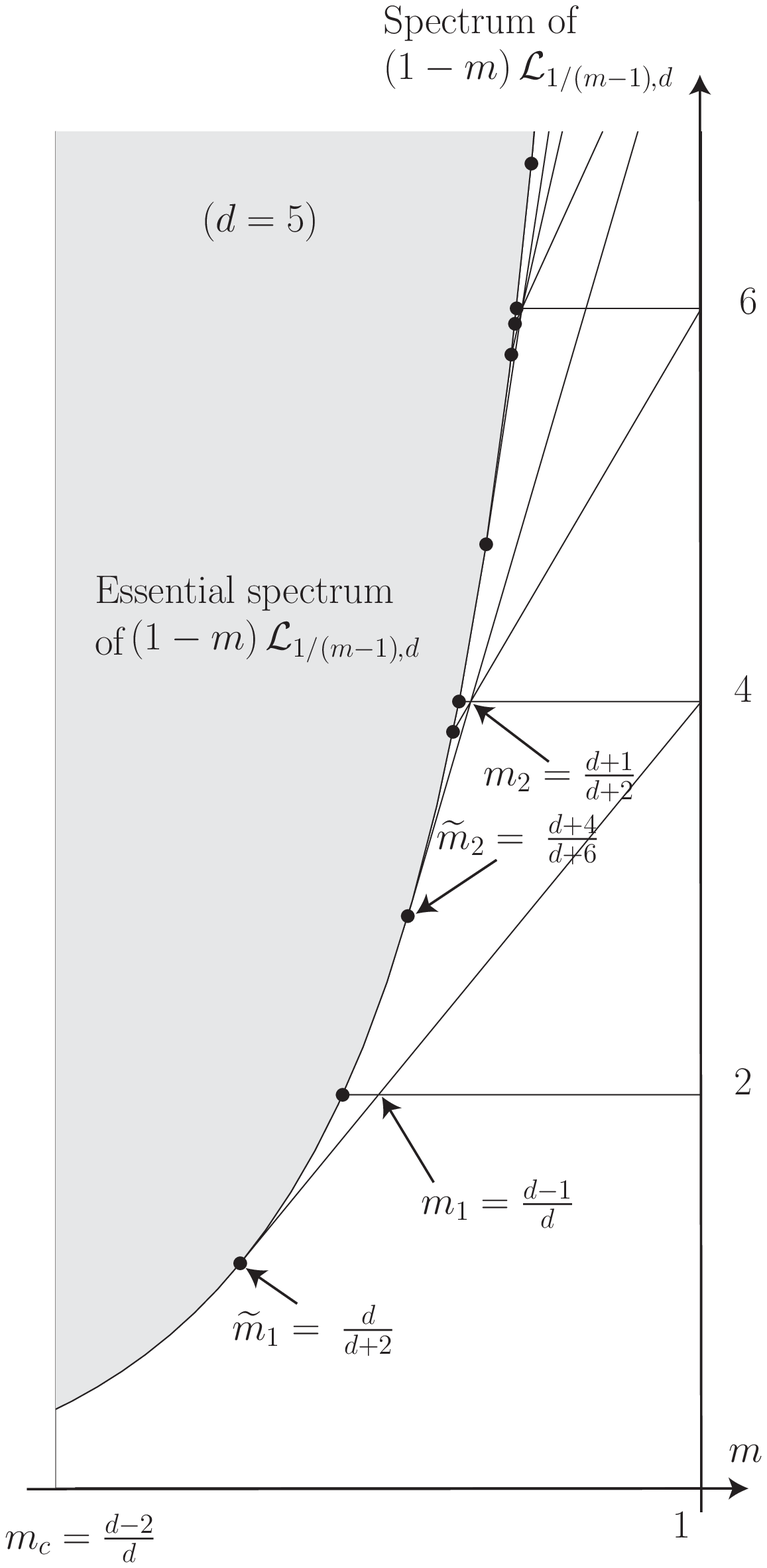}\caption{Spectrum of $\mathcal L_{\alpha,d}$ as a function of $\alpha$ (left),
and spectrum of $(1-m)\,\mathcal L_{1/(m-1),d}$ as a function of~$m$ (right), for $d=5$.}\end{center}\end{figure}

A crucial observation is that we can scale Inequality \eqref{gap}.
\begin{corollary}\label{Cor:LinearizedProblem}
Let $M>0$ and $d\ge2$. With $B_\sigma$ defined by \eqref{Eqn:Barenblatt}, let $f$ be a function in $L^2(B_\sigma^{2-m}\,dx)$ such that
\[
\ird{(1,x,|x|^2)\,f\,B_\sigma^{2-m}}=(0,0,0)\quad\mbox{and}\quad \nabla f\in L^2(B_\sigma\,dx)\;.
\]
Then the inequality
\[
\Lambda_{\alpha,d}\int_{\R^d}|f|^2\,B_\sigma^{2-m}\;dx\le\sigma^{\frac d2(m-m_c)}\,\int_{\R^d}|\nabla f|^2\,B_\sigma\;dx
\]
holds with $\Lambda_{\alpha,d}$ as in Corollary~\ref{Cor:ImprovedHardyPoincare}. \end{corollary}
\begin{proof} The proof relies on a simple change of variables. From Corollary~\ref{Cor:ImprovedHardyPoincare}, we know that
\[
\Lambda_{\alpha,d}\int_{\R^d}|g|^2\,\(1+|x|^2\)^{\alpha-1}\,dx\le\int_{\R^d}|\nabla g|^2\,\(1+|x|^2\)^\alpha\,dx
\]
for any $g\in L^2\,(d\mu_{\alpha-1})$ satisfying the conditions of Corollary~\ref{Cor:ImprovedHardyPoincare}. Then
Corollary~\ref{Cor:LinearizedProblem} holds for $f$ such that $f(x)=g(x/\sqrt{C_M\,\sigma}\,)$ for any $x\in\R^d$, which concludes the
proof.\end{proof}

\section{Estimates on the second moment}\label{Sec:SecondMoment}

Up to now, we have not determined the behavior of $R(\tau)$ as $\tau\to\infty$, nor the fact that $\sigma$ has a finite, positive limit as
$t\to\infty$. These properties can be deduced for instance from \cite{BBDGV}. For the convenience of the reader, let us give some details. As in
\cite{BBDGV}, consider the standard change of variables
\[
(\tau,y)\mapsto\(t=\tfrac 12\,\log R_0(\tau),x=\tfrac y{R_0(\tau)}\)
\]
based on the self-similar behavior of the Barenblatt solution $\mathsf B(\tau,y)$, where
\[
R_0(\tau):=\(1+2\,d\,(m-m_c)\,\tau\)^\frac 1{d\,(m-m_c)}\;.
\]
With the above change of variables, if $v$ is a solution of \eqref{Eqn}, then the function $U(t,x):=v(\tau,y)$ solves
\[
\frac{\partial U}{\partial t}+\nabla\cdot\left[U\,\(\nabla U^{m-1}-2\,x\)\right]=0\quad t>0\;,\quad x\in\R^d\,,
\]
which is nothing else than \eqref{Eqn2}, except that here $\sigma$ is replaced by $1$. It has been established in \cite{BBDGV} that
\[
\lim_{t\to\infty}e^{2\,\mathsf c\,t}\,\mathcal F_1[U(t,\cdot)]=0
\]
for some positive constant $\mathsf c$. By H\"older's inequality, we know that
\begin{multline*}
\ird{U^m}=\ird{\(U\,B_1^{m-1}\)^m\,B_1^{m\,(1-m)}}\\
\le\(\ird{U\,B_1^{m-1}}\)^m\,\(\ird{B_1^m}\)^{1-m}\,,
\end{multline*}
thus proving that
\begin{multline*}
\mathcal F_1[U]\ge\ird{B_1^m}+\frac m{1-m}\ird{U\,B_1^{m-1}}\\
-\frac 1{1-m}\(\ird{U\,B_1^{m-1}}\)^m\,\(\ird{B_1^m}\)^{1-m}\,.
\end{multline*}
With
\[
\mathsf k(t):=\frac{\ird{U\,B_1^{m-1}}}{ \ird{B_1^{m}}}=\frac{\ird{|x|^2\,U}+M\,C_M}{K_M+M\,C_M}\;,
\]
this can be rewritten as
\[
\mathcal F_1[U]\ge\ird{B_1^m}\,\frac{\mathsf k^m-1-m\,(\mathsf k-1)}{m-1}\sim|\mathsf k(t)-1|^2\quad\mbox{as}\;t\to\infty\;,
\]
hence showing
\[
\lim_{t\to\infty}e^{\mathsf c\,t}\,|\mathsf k(t)-1|=0\;.
\]
As a consequence, we observe that
\[
\left|\ird{|x|^2\,U(t,x)}-K_M\right|=O\big(e^{-\mathsf c\,t}\big)\quad\mbox{as}\;t\to\infty\;,
\]
Undoing the change of variables, we find that
\[
\left|\frac 1{R_0(\tau)^2}\int_{\R^d}|y|^2\,v(\tau,y)\;dy-K_M\right|=O\big(R_0(\tau)^{-\mathsf c/2}\big)\quad\mbox{as}\;\tau\to\infty\;,
\]
which, by definition of $\sigma$, gives
\[
\sigma=\frac 1{R(\tau)^2}\,\(1+2\,d\,(m-m_c)\,\tau\)^\frac{2}{d\,(m-m_c)}\left(1+O\big(R_0(\tau)^{-\mathsf c/2}\big)\right)
\]
as $\tau\to\infty$, where $\tau\mapsto R(\tau)$ is given by \eqref{Eqn:Scale}. Using \eqref{Eqn:sigma}, this means that
\[
\frac 1{R(\tau)}\,\frac{dR}{d\tau}=2\,(R_0(\tau)^2\,\sigma)^{-\frac d2(m-m_c)}=\frac 1{d\,(m-m_c)}\,\frac 1\tau\,\left(1+O\big(\tau^{-\varepsilon}\big)\right)
\]
with $\varepsilon=\min\{1,\mathsf c/2\}>0$. With these estimates, we can prove the following result.
\begin{lemma}\label{Lem:Asymptotic} With the notations of Sections~\ref{Sec:Intro} and \ref{Sec:RelativeEntropy},
\[
R(\tau)\sim\tau^{\frac 1{d\,(m-m_c)}}\quad\mbox{as}\;\tau\to\infty
\]
and, as a function of $t$, $t\mapsto\sigma(t)$ is positive, decreasing, with
\[
\lim_{t\to\infty}\sigma(t)=:\sigma_\infty>0\;.
\]\end{lemma}
More precisely, we know that for some $C_\infty>0$, 
\[
R(\tau)=C_\infty\,\tau^{\frac 1{d\,(m-m_c)}}\left(1+o(1)\right)\quad\mbox{as}\;\tau\to\infty
\]
and $\sigma_\infty={(2\,d\,(m-m_c))^{\frac 2{d\,(m-m_c)}}}/{C_\infty^2}$. However, the value of $\sigma_\infty$ in terms of $v_0$ is not known.
\begin{proof}
The asymptotic behaviors of $R$ and $\sigma$ are direct consequences of the above computations. We only have to prove the monotonicity of
$\sigma$. According to Lemma~\ref{Lem:ChoiceOfSigma}, we know that
\[
\sigma(t)=\frac 1{K_M}\ird{|x|^2\,u(t,x)}\;.
\]
Using \eqref{Eqn2} and integrating by parts, we compute
\begin{multline*}
\frac{d\sigma}{dt}=\frac 1{K_M}\,\frac d{dt}\ird{|x|^2\,u(t,x)}\\
=\frac{2\,d\,(1-m)}{m\,K_M}\,\sigma^{\frac d2(m-m_c)}\ird{u^m}-4\,\sigma\;.
\end{multline*}
On the other hand, by Lemma~\ref{Lem:ChoiceOfSigma}, we have
\[
\ird{B_{\sigma(t)}^{m-1}\,u(t,x)}=\ird{B_{\sigma(t)}^m}\;,
\]
so that
\begin{multline*}
\mathcal F_{\sigma(t)}[u(t,\cdot)]=\ird{\frac{u^m-B_{\sigma(t)}^m}{m-1}}\\=
\frac{2\,m}{d\,(1-m)^2}\,K_M\,\sigma^{\frac d2(1-m)}-\frac 1{1-m}\ird{u^m}\;.
\end{multline*}
Hence we have proved that
\[
\frac{d\sigma}{dt}=-2\,d\,\frac{(1-m)^2}{m\,K_M}\,\sigma^{\frac d2(m-m_c)}\,\mathcal F_{\sigma(t)}[u(t,\cdot)]\le 0\;,
\]
which completes the proof. \end{proof}

\begin{remark}\label{Rem3} We may notice that the second moment converges to a well defined value:
\[
\lim_{t\to\infty}\ird{|x|^2\,u(t,x)}=K_M\,\sigma_\infty\;.
\]
This is not the one which is usually found by considering the time-dependent rescaling corresponding to the self-similar Barenblatt
solutions.\end{remark}

\section{Proof of Theorem~\ref{Thm:Main}}\label{Sec:Proof}

We are now ready to resume with the relative entropy estimates and conclude the proof of Theorem~\ref{Thm:Main}. Using \eqref{Eqn:Entropy},
\eqref{Eqn:Fisher} and Corollary~\ref {Cor:LinearizedProblem}, we find as in \cite{BDGV} that \be{Ineq:Interpolation} \mathcal F_\sigma[u]\le
\frac{h^{2-m}\,[1+X(h)]}{2\,\big[\Lambda_{\alpha,d}-\sigma\,Y(h)\big]}\,m\,\sigma^{\frac d2(m-m_c)}\,\mathcal I_\sigma [u] \ee as soon as
$0<h<h_*:=\min\{h>0\,:\,\Lambda_{\alpha,d}-\sup_{t\in\R^+}\sigma(t)\,Y(h)\ge 0\}$. Two differences with \cite{BDGV} arise: $\Lambda_{\alpha,d}$
has been improved in Corollary~\ref {Cor:LinearizedProblem}, to the price of a factor~$\sigma$, which however plays no role because it also
appears in the computation of $\frac d{dt}\mathcal F_{\sigma(t)}[u(t,\cdot)]$. The $m$ factor is present because $\mathcal F_\sigma$ has not been
normalized as in \cite{BDGV}, and also because the equation for $v$ is not the same.

As in \cite{BDGV}, uniform relative estimates hold, according to \cite[Inequality (5.33)]{BGV}: for some positive constant $\mathsf C$, we have
\[
0\le h(t)-1\le\mathsf C\,\mathcal F_{\sigma(t)}[u(t,\cdot)]^\frac{1-m}{d+2-(d+1)m}\,.
\]
Summarizing, we end up with a system of nonlinear differential inequalities, with $h$ as above and, at least for some $t_*>0$ large enough,
\[
\frac{d}{dt}\mathcal F_{\sigma(t)}[u(t,\cdot)] \le-\,2\,\frac{\Lambda_{\alpha,d}-\sigma(t)\,Y(h)}{\big[1+X(h)\big]\,h^{2-m}}\,(1-m)\,\mathcal
F_{\sigma(t)}[u(t,\cdot)]
\]
 for any $t>t_*$. Gronwall type estimates then show that
\[
\limsup_{t\to\infty}\,{\rm e}^{2\,(1-m)\,\Lambda_{1/(m-1),d}\,t}\mathcal F_{\sigma(t)}[u(t,\cdot)]<\infty\;.
\]
Notice as in \cite{BBDGV} that for some constant $c>0$, $\lim_{t\to\infty}e^{c\,t}(h(t)-1)=0$, so that the fact that the quotient $\mathcal
I_\sigma [u]/\mathcal F_\sigma[u]$ in \eqref{Ineq:Interpolation} is not estimated exactly by $(2\,\Lambda_{\alpha,d})/(m\,\sigma)$ plays no role
for the rate of convergence. This completes the proof of Theorem~\ref{Thm:Main} with $\gamma(m)=(1-m)\,\Lambda_{1/(m-1),d}$ (see Fig.~2).\finprf

\begin{remark}\label{Rem4}
Exactly as in \cite[Corollary 3]{BDGV}, explicit estimates of the constants can be obtained. If $v_0/B_1$ is uniformly bounded from above and from
below by two positive constants, an estimate of $\limsup_{\tau\to\infty}R(\tau)^{\gamma(m)}\,\mathcal F_{\sigma(t)}[u]$ in Theorem~\ref{Thm:Main}
can be given in terms of $v_0$ by computing a Gronwall estimate.
\end{remark}

\begin{figure}[hb]\begin{center}\includegraphics[width=12cm]{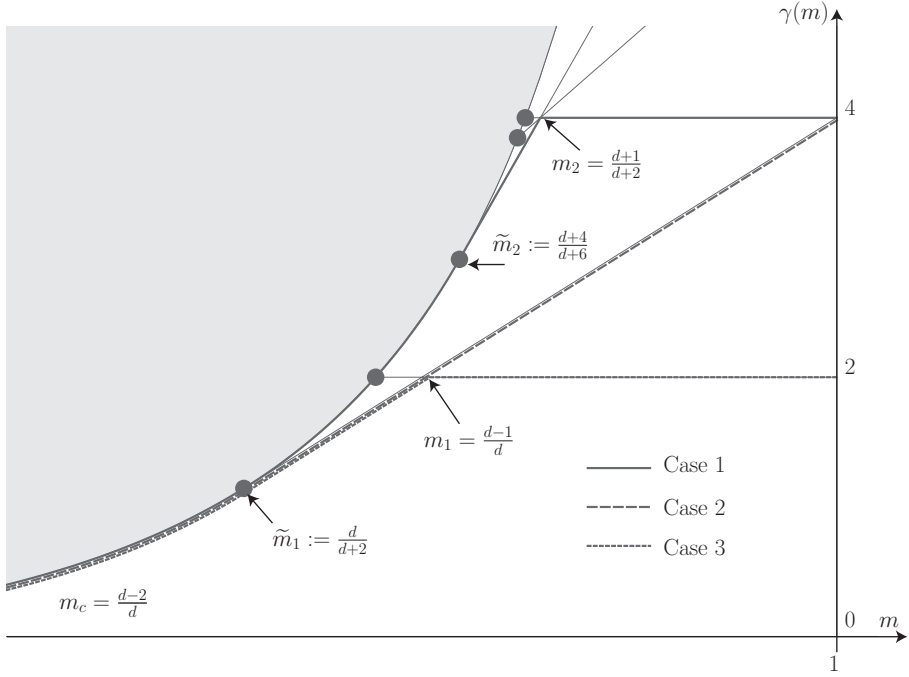}
\caption{For $d=5$, the value of $m\mapsto\gamma(m)$ is given by the curve of Case 1 when no
 assumptions are made on the initial data. The curve of Case 2 corresponds to the case
  where the center of mass is chosen at the origin, as in~\cite{BDGV}, while the curve of Case 3
  corresponds to the exponent found in Theorem~\ref{Thm:Main}.}\end{center}\end{figure}

\section{Concluding remarks}\label{Sec:Conclusion}

The case of the heat equation, \emph{i.e.} $m=1$, is not covered in our approach. However, we may pass to the limit as $m\to 1_-$ in
Corollary~\ref{Cor:LinearizedProblem}, for the special choice $1/\sigma=2\,(1-m)$. Both weights $B_\sigma^{2-m}$ and $B_\sigma$ converge to the Gaussian weight, so that the conditions of Corollary~\ref{Cor:LinearizedProblem} become
\[
\ird{(1,x,|x|^2)\,f\;e^{-|x|^2/2}}=(0,0,0)\quad\mbox{and}\quad\nabla f\in L^2(e^{-|x|^2/2}\,dx)\;.
\]
By requiring the orthogonality with respect to all Hermite polynomials up to order two, we achieve the improved Poincar\'e inequality
\[
2\int_{\R^d}|f|^2\;e^{-|x|^2/2}\;dx\le \int_{\R^d}|\nabla f|^2\;e^{-|x|^2/2}\;dx\;.
\]
Compared with the results of \cite{BDGV}, nothing is gained, as
\[
\lim_{m\to 1_-}(1-m)\,\lambda_{01}=\lim_{m\to 1_-}(1-m)\,\lambda_{20}=4\;,
\]
where $\lambda_{\ell k}$ are defined in Proposition~\ref{Prop:Spectrum} and $\alpha=1/(m-1)$. See \cite{BBDE} for more details on improved
convergence rates of relative entropies in case $m=1$.

\medskip
For completeness, let us extend our results to the case $d=1$, which is very simple. Eigenvalues of $\mathcal L_{\alpha,1}$ are ordered uniformly
with respect to $\alpha$, according to Proposition~\ref{Prop:Spectrum}. Let $\alpha<-1/2$ and consider $f\in L^2\,(d\mu_{\alpha-1})$ such that
\[
\int_\R f\,d\mu_{\alpha-1}=0\;,\quad\int_\R x\,f\,d\mu_{\alpha-1}=0\quad\mbox{and}\quad\int_\R |x|^2\,f\,d\mu_{\alpha-1}=0\;.
\]
Then \eqref{gap} holds with
\[
\Lambda_{\alpha,d}=\left\{\begin{array}{ll}
(\alpha-\tfrac 12)^2&\mbox{if}\;\alpha\in\left[-\tfrac 52,-\tfrac 12\right)\;,\vspace{6pt}\cr
-\,6\,(\alpha+1))&\mbox{if}\;\alpha\in(-\infty,-\tfrac 52]\,.
\end{array}\right.
\]

\begin{theorem}\label{Thm:MainOneD} Assume that $d=1$, $m\in(1/3,1)$.
Let $v$ be a solution of~\eqref{Eqn} with initial datum $v_0\in L^1_+(\R)$ such that $v_0^m$ and $|y|^2\,v_0$ are integrable. With the above
notations, we have $\limsup_{\tau\to\infty}R(\tau)^{\gamma(m)}\,\mathcal E(\tau)<\infty$, where $R(\tau)\sim\tau^{\frac 1{m+1}}$ as
$\tau\to\infty$ and
\[
\gamma(m)=\left\{\begin{array}{ll}
\frac{(3-m)^2}{4\,(1-m)}&\mbox{if}\;m\in\left(\tfrac 13,\tfrac 35\right]\;,\vspace{6pt}\cr
6\,m&\mbox{if}\;m\in\left[\tfrac 35,1\right)\;,\cr
\end{array}\right.
\]
\end{theorem}

\bigskip\noindent{\it Acknowledgments.\/}{ \small The authors acknowledge support both by the ANR-08-BLAN-0333-01 project CBDif-Fr (JD) and by MIUR project ``Optimal mass transportation, geometrical and functional inequalities with applications'' (GT). The warm hospitality of the laboratory Ceremade of the University of Paris-Dauphine, where this work have been partially done, is kindly acknowledged. The authors thank two anonymous referees for their careful reading of the paper.}

\par\noindent{\scriptsize\copyright\,2011 by the authors. This paper may be reproduced, in its entirety, for non-commercial purposes.}


\def\cprime{$'$}

\end{document}